\documentclass[oneside,english]{amsart}
\usepackage[T1]{fontenc}
\usepackage[latin9]{inputenc}
\usepackage{geometry}
\usepackage{amsthm}
\usepackage{amssymb}

\makeatletter
\numberwithin{equation}{section}
\numberwithin{figure}{section}
\theoremstyle{plain}
\newtheorem{thm}{\protect\theoremname}[section]
  \theoremstyle{plain}
  \newtheorem{assumption}[thm]{\protect\assumptionname}
  \theoremstyle{remark}
  \newtheorem{rem}[thm]{\protect\remarkname}
  \theoremstyle{plain}
  \newtheorem{lem}[thm]{\protect\lemmaname}
  \theoremstyle{plain}
  \newtheorem{cor}[thm]{\protect\corollaryname}
  \theoremstyle{definition}
  \newtheorem{defn}[thm]{\protect\definitionname}
  \theoremstyle{definition}
  \newtheorem{example}[thm]{\protect\examplename}
  \theoremstyle{remark}
  \newtheorem*{acknowledgement*}{\protect\acknowledgementname}

\makeatother

\usepackage{babel}
  \providecommand{\acknowledgementname}{Acknowledgement}
  \providecommand{\assumptionname}{Assumption}
  \providecommand{\corollaryname}{Corollary}
  \providecommand{\definitionname}{Definition}
  \providecommand{\examplename}{Example}
  \providecommand{\lemmaname}{Lemma}
  \providecommand{\remarkname}{Remark}
\providecommand{\theoremname}{Theorem}

\begin{document}

\title{Tight Bounds for Averaging Multi-Frequency Differential Inclusions,
Applied to Control Systems}

\author{Ido Bright}

\address{Department of Applied Mathematics, University of Washington, Seattle,
Washington, USA.}

\email{ibright@uw.edu}

\maketitle
\global\long\def\seq#1{#1_{1},#1_{2},\dots}

\global\long\def\seqm#1{#1_{1},#1_{2},\dots,#1_{m}}

\global\long\def\seqf#1{#1_{1}\left(\cdot\right),#1_{2}\left(\cdot\right),\dots}

\global\long\def\seqfn#1#2{#1_{1}\left(\cdot\right),#1_{2}\left(\cdot\right),\dots,#1_{#2}\left(\cdot\right)}

\global\long\def\seqn#1#2{#1_{1},#1_{2},\dots,#1_{#2}}

\global\long\def\de{\Delta\left(\epsilon\right)}

\global\long\def\dej{\Delta_{j}\left(\epsilon\right)}

\global\long\def\ee{\eta\left(\epsilon\right)}

\global\long\def\smj{\sum_{j=1}^{m}}

\global\long\def\eej{\eta_{j}\left(\epsilon\right)}

\global\long\def\e{\epsilon}

\global\long\def\D{\Omega}

\global\long\def\pd{\left(\cdot\right)}

\global\long\def\ze{\left[0,\e^{-1}\right]}

\global\long\def\S#1{S_{\ze}\left(\e#1,x_{0}\right)}

\global\long\def\Ss#1{S_{\ze}\left(#1,x_{0}\right)}

\global\long\def\SF{\S F}

\global\long\def\SFT{\S{F_{T}}}

\global\long\def\SFB{\S{\bar{F}}}

\global\long\def\iT{\left[iT,\left(i+1\right)T\right)}

\global\long\def\iTj{\left[iT_{j},\left(i+1\right)T_{j}\right)}

\global\long\def\Fxt{F\left(t,x\right)}

\global\long\def\x{x^{*}\left(\cdot\right)}

\global\long\def\xx#1{x^{*}\left(#1\right)}

\global\long\def\dxx#1{\dot{x}^{*}\left(#1\right)}

\global\long\def\y{y^{*}\left(\cdot\right)}

\global\long\def\fxt{f\left(t,x,u\right)}

\global\long\def\t{\tau}
\global\long\def\tn{\tau_{n}}
\global\long\def\ti{\tau_{i}}
\global\long\def\tii{\tau_{i+1}}

\global\long\def\Fj{F_{j}\left(t,x\right)}

\global\long\def\R{\mathbb{R}}

\global\long\def\Rd{\mathbb{R}^{d}}

\global\long\def\Fbx{\overline{F}\left(x\right)}

\global\long\def\jm{j=1,\dots,m}

\global\long\def\Fx#1{F\left(#1,x\right)}

\global\long\def\Fb#1{\bar{F}\left(#1\right)}

\global\long\def\rinf#1{#1\rightarrow\infty}

\global\long\def\linf#1{\lim_{#1\rightarrow\infty}}

\global\long\def\liinf#1{\liminf_{#1\rightarrow\infty}}

\global\long\def\lsinf#1{\limsup_{#1\rightarrow\infty}}

\begin{abstract}
We present new tight bounds for averaging differential inclusions,
which we apply to multi-frequency inclusions consisting of a sum of
time periodic set-valued mappings. For this family of inclusions we
establish an a tight estimate of order $O\left(\epsilon\right)$ on
the approximation error. These results are then applied to control
systems consisting of a sum of time-periodic functions.
\end{abstract}

\section{Introduction}

The averaging of differential inclusions seeks to approximate the
solution-set of a time varying differential inclusion with small amplitude
(or, equivalently by change of variable, a highly oscillatory systems),
by the solution of the auxiliary \textit{averaged} autonomous differential
inclusion, in a finite but large time domain. The averaged inclusion
is obtained by computing the time average of the set-valued mapping.
As a time-independent inclusion it is amenable to analysis, and applications
of averaging in stabilization and optimality can be found in Gama
and Smirnov \cite{gama2013stability}. In this paper we focus on estimating
the difference, in the Hausdorff distance, between the solution sets
of both systems. 

We consider the quantitative aspect of the approximation of $\SF$,
the solution-set in $\ze$ of the differential inclusion
\begin{equation}
\dot{x}\in\e\Fxt,\; x\left(0\right)=x_{0},\label{eq:diffinc}
\end{equation}
where we focus on the case where
\[
\Fxt=F_{1}\left(\omega_{1}t,x\right)+\cdots+F_{m}\left(\omega_{m}t,x\right)
\]
and each $\Fj$ is periodic in $t$ with period 1. The solution set
is approximated by $\SFB$ the solution-set of the averaged differential
inclusion 
\begin{equation}
\dot{y}\in\epsilon\bar{F}\left(y\right),\; y\left(0\right)=x_{0}\label{eq:avginc}
\end{equation}
in $\ze$, where

\begin{equation}
\bar{F}(x)=\lim_{T\rightarrow\infty}\frac{1}{T}\int_{0}^{T}F(s,x)ds.\label{eq:Average def}
\end{equation}
The integral considered is the Aumann integral (see, Aumann \cite{aumann1965integrals}),
and the convergence is in the Hausdoeff distance.

Our main result establishes an $O\left(\e\right)$ estimation of the
approximation error, i.e., the Housdorff distance between the solution
sets $\SF$ and $\SFB$. Namely, for every solution of \eqref{eq:diffinc}
there is a solution of \eqref{eq:avginc} which is $O\left(\e\right)$
close to it, and vice-versa. This result extends the classical bound
of $O\left(\e\right)$ for time periodic inclusion ($m=1$) by Plotnikov
\cite{plotnikov1979averaging} to multi-frequency inclusions. We also
establish the same tight estimate when each coordinate of $\Fxt$
is periodic with a different period, and provide new estimates for
general ,non-periodic, differential inclusions.

These results are applied to the averaging of control systems of the
form
\begin{equation}
\dot{x}=\e g\left(t,x,u\right),\ \ x\left(0\right)=x_{0}\label{eq:control_sys}
\end{equation}
where
\[
g\left(t,x,u\right)=g_{1}\left(\omega_{1}t,x,u\right)+\cdots+g_{m}\left(\omega_{m}t,x,u\right)
\]
and each $g_{i}\left(t,x,u\right)$ is periodic in $t$ with period
$1$. The difficulty in this setting lies in the fact that same control
appears in all terms, and a non-trivial extension of our results is
presented in Section \ref{sec:Application-in-multi-frequency}. The
averaged equation corresponds to the chattering limit
\begin{equation}
\dot{y}\in\e\bar{G}(y)=\lim_{T\rightarrow\infty}\frac{\e}{T}\int_{0}^{T}\left\{ g(s,y,u)|u\in U\right\} ds,\label{eq:cnt_avg-1}
\end{equation}
and not to the trivial time average. An equivalent definition of $\bar{G}\left(y\right)$,
which replaces the time average by a space average, can be obtained
when all the $g_{j}\left(t,x,u\right)$ are continuous in $t$ and
the set of frequencies $\omega_{1},\omega_{2},\dots,\omega_{m}$ is
linearly independent over the integers. Then 
\begin{equation}
\bar{G}\left(y\right)=\left\{ \int_{\left[0,1\right]^{m}}\smj g_{j}\left(\phi_{j},y,u\left(\phi\right)\right)d\phi|u:\left[0,1\right]^{m}\rightarrow U\mbox{ is measurable}\right\} ,\label{eq:cont_avg_meas_fnc}
\end{equation}
where $\phi=\left(\phi_{1},\phi_{2},\dots,\phi_{m}\right)\in\left[0,1\right]^{m}$.

We establish that the approximation error for this multi-frequency
system is $O\left(\e\right)$, extending the result in \cite{plotnikov1979averaging}
as well as that by Bombrun and Pomet \cite[Theorem 3.7]{bombrun2013averaged}
for linear systems. This bound also improves a previous bound of order
$O\left(\sqrt{\e}\right)$ presented by Artstein, for the case $m>1$.

Applying a change of variables $\tau=\e^{-1}t$ Equations \eqref{eq:diffinc}
and \eqref{eq:avginc} reduce to 
\[
x'\in F\left(\tau/\e,x\right),\ y'\in\bar{F}\left(y\right)\ x\left(0\right)=y\left(0\right)=x_{0},
\]
and Equations \eqref{eq:control_sys} and \eqref{eq:cnt_avg-1} to
\[
x'=g\left(\tau/\e,x,u\right),\ y'\in\bar{G}\left(y\right)\ x\left(0\right)=y\left(0\right)=x_{0}.
\]
With this change of variable, our bounds hold in the time interval
$\left[0,1\right]$.

Averaging differential inclusions generalizes the classical averaging
method of averaging differential equations. For a reference on the
averaging method of ordinary differential equations, the reader is
referred to the book of Sanders, Verhulst and Morduck \cite{sanders2007averaging}
and to works of Artstein \cite{artstein2007averaging} and Bright
\cite{bright2009tight,bright2011moving} for a modern treaty and improved
estimates on the error, the line of which we follow in this paper.
For a reference to results in differential inclusions refer to the
review papers Klymchuk, Plotnikov and Skripnik \cite{klymchuk2012overview}
and to \cite{gama2013stability}.

The estimates presented in this paper are one of the first qualitative
results for averaging differential inclusions. We believe that the
methods presented may also be applied to quantitate analysis of averaging
singularly perturbed differential inclusions and control systems,
where quantitative bounds are sparse. For a review on averaging in
singularly perturbed control systems see Artstein \cite{artsteinanalysis}
and the references within.

The structure of this paper is as follows: Section \ref{sec:Notations-and-Assumptions}
presents the assumption and notations used throughout this paper.
Section \ref{sec:Key-Lemma} presents our key lemma, which estimates
the effect of averaging over a finite interval on the solution sets.
 In Section \ref{sec:Averaging-differential-Inclusion} results averaging
results of differential inclusions are presented, and in the last
section they are applied to control systems.

\section{\label{sec:Notations-and-Assumptions}Notations and Assumptions}

In what follows, we use the following notions. We denote the d-dimensional
Euclidean space by $\R^{d}$, a vector by $x\in\R^{d}$ and its Euclidean
norm by $\left|x\right|$. The Euclidean ball centered at $x$ with
radius $r>0$ is denoted by $B\left(x,r\right)\subset\Rd$. Given
two sets $A_{1},A_{2}\subset\R^{d}$ their Minkovski sum is denoted
by $A_{1}+A_{2}=\left\{ x_{1}+x_{2}|x_{1}\in A_{1},x_{2}\in A_{2}\right\} $.
We endow the set of continuous function by supremum norm, defined
by $\left\Vert y\pd\right\Vert =\sup_{t}\left|y\left(t\right)\right|$.
Given a normed vector space $\left(\mathcal{X},\left\Vert \cdot\right\Vert \right)$
the distance between a point $x\in\mathcal{X}$ and a set $A\subset\mathcal{X}$
is denoted by $d\left(x,A\right)=\inf\left\{ \left\Vert x-y\right\Vert |y\in A\right\} $
and the Hausdorff distance between two sets $A_{1},A_{2}\in\mathcal{X}$
by 
\[
d_{H}\left(A_{1},A_{2}\right)=\max\left\{ \sup\left\{ d\left(y,A_{2}\right)|y\in A_{1}\right\} ,\sup\left\{ d\left(y,A_{1}\right)|y\in A_{2}\right\} \right\} .
\]

We define the support function of a convex set $D\subset\Rd$ by $h_{D}\left(x\right)=\sup_{y\in D}x\cdot y$
for every $x\in\partial B\left(\mathbf{0},1\right)$. When $D_{1},D_{2}\subset\Rd$
are convex their Hausdorff distance (see, Schneider \cite[Theorem 1.8.11]{schneider1993convex})
is given by
\[
d_{H}\left(D_{1},D_{2}\right)=\left\Vert h_{D_{1}}\pd-h_{D_{2}}\pd\right\Vert .
\]

We use the notation $\dot{x}\left(t\right)=\frac{d}{dt}x\left(t\right)$
for the time derivative. The solution-set of the differential inclusion
$\dot{x}\in G\left(t,x\right)$ with initial condition $x\left(0\right)=x_{0}$
and in the domain $\left[0,T\right]$, is denoted by $S_{\left[0,T\right]}\left(G,x_{0}\right)$. 

We consider solutions of differential equations of the form $\dot{x}\in\e\Fxt$
in the domain $\D\subset\Rd$, satisfying the following conditions.
\begin{assumption}
\label{ass:Averaging condition}The set-valued mapping $F(x,t):\R\times\D\rightrightarrows\mathbb{R}^{d}$
satisfies the following conditions: \end{assumption}
\begin{enumerate}
\item The values of $\Fxt$ are non-empty, closed and convex in its domain.
\item For every $t,x$ in its domain $\Fxt\subset B\left(0,M\right)$.
\item $\Fxt$ is measurable in $t$.
\item $\Fxt$ is uniformly Lipschitz continuous in $x$ with a Lipschitz
constant $K$, namely,
\[
d_{H}\left(F\left(t,x_{1}\right),F\left(t,x_{2}\right)\right)\le K\left|x_{1}-x_{2}\right|
\]
for all $t\in\R$, $x_{1},x_{2}\in\D$.
\item The time average function $\bar{F}\left(x\right)$, defined in \eqref{eq:Average def},
exists.
\end{enumerate}
Throughout this paper we assume the following assumption on the solutions
of \eqref{eq:diffinc} and \eqref{eq:avginc}.
\begin{assumption}
All the solutions of \eqref{eq:diffinc} and \eqref{eq:avginc} are
contained in $\D$.
\end{assumption}
Without loss of generality and to simplify our proofs, we shall assume
that $\Fxt$ satisfies the conditions of Assumption \ref{ass:Averaging condition}
for $\D=\Rd$.

Our assumptions can be relaxed in the following manner.
\begin{rem}
\label{rem: non-cnv}The requirement that $\Fxt$ is convex valued
can be relaxed. Indeed, by Filippov theorem the solution set of $\Fxt$
is dense in the solution set of the inclusion obtained by replacing
$\Fxt$ by its convex hull. Moreover, the average of both set-valued
mappings are equal.
\end{rem}

\begin{rem}
The Lipschitz regularity of $\Fxt$ can be relaxed so that the Lipschitz
constant, $k\left(t\right)$, depends on $t$. In this case the same
results hold as long as there exists $K$ so that $\e\int_{0}^{T}k\left(s\right)ds\le K$
holds for every $T\in\ze$.
\end{rem}

\begin{rem}
The solutions are presented in a finite dimensional Euclidean space,
however, they hold for Banach-valued differential inclusions as well. 
\end{rem}
The following lemma easily follows from the assumptions above.
\begin{lem}
If $\Fxt$ satisfies Assumption \ref{ass:Averaging condition} then
so does $ $$\bar{F}(x)$.
\end{lem}
Assumption \ref{ass:Averaging condition} implies the existence of
Filippov solutions to both \eqref{eq:diffinc} and \eqref{eq:avginc}
in any finite time interval, as well as the validity of the Filippov-Gronwall
inequality stated below.
\begin{thm}
\label{thm:Filippov lemma}Let $\Fxt$ satisfy the conditions of Assumption
\ref{ass:Averaging condition} and $y:\left[0,T\right]\rightarrow\D$
be an absolutely continuous function satisfying $y\left(0\right)=x_{0}$.
There exist a solution $x^{*}\left(\cdot\right)$ of \eqref{eq:diffinc}
such that 
\[
\sup_{t\in\left[0,T\right]}\left|x^{*}\left(t\right)-y\left(t\right)\right|\le e^{\epsilon KT}\int_{0}^{T}d\left(\dot{y}\left(s\right),\e F\left(s,y\left(s\right)\right)\right)ds.
\]

\end{thm}
We shall use the following corollary.
\begin{cor}
\label{cor:filippov}Suppose $F_{1}\left(t,x\right)$ and $ $$F_{2}\left(t,x\right)$
satisfy the conditions of Assumption \ref{ass:Averaging condition}
for $\D=\Rd$, and that $d_{H}\left(F_{1}\left(t,x\right),F_{2}\left(t,x\right)\right)<\eta$
holds for for every $t>0$ and $ $ $x\in\D$ then
\[
d_{H}\left(\S{F_{1}},\S{F_{2}}\right)<e^{K}\eta.
\]
\end{cor}
\begin{proof}
Let $x_{1}^{*}\left(\cdot\right)$ be an arbitrary solution of $\dot{x}_{1}\in\e F_{1}\left(t,x_{1}\right)$
defined in $\ze$. Then by the Filippov-Gronwall inequality there
exists a solution $x_{2}^{*}\pd$ of $\dot{x}_{2}\in\e F_{2}\left(t,x_{2}\right)$
satisfying $x_{1}^{*}\left(0\right)=x_{2}^{*}\left(0\right)$ and
\begin{eqnarray*}
\sup_{t\in\ze}\left|x_{1}^{*}\left(t\right)-x_{2}^{*}\left(t\right)\right| & \le & e^{K}\int_{0}^{\e^{-1}}d\left(\dot{x}_{1}^{*}\left(s\right),\e F_{2}\left(s,x_{1}^{*}\left(s\right)\right)\right)ds\\
 & \le & e^{K}\int_{0}^{\e^{-1}}\e d_{H}\left(F_{1}\left(s,x_{1}^{*}\left(s\right)\right),F_{2}\left(s,x_{1}^{*}\left(s\right)\right)\right)ds\le e^{K}\int_{0}^{\e^{-1}}\e\eta ds=e^{K}\eta.
\end{eqnarray*}
The other direction is equivalent.
\end{proof}

\section{\label{sec:Key-Lemma}Key Lemma}

In this section we study the effect of a finite-time averaging, or
partial average, on the solution-set of a differential inclusion.
The bound we obtain is used in the following section to estimate the
averaging approximation error, where we apply  such finite-time averages
in a sequential manner. 
\begin{defn}
Given a set-valued mapping $\Fx t$ and $T>0$ we define
\[
F_{T}\left(t,x\right)=\frac{1}{T}\int_{0}^{T}F\left(t+s,x\right)ds.
\]

\end{defn}
Notice that when $\Fxt$ is periodic in $t$ with period $T$ then
$ $$F_{T}\left(t,x\right)=\Fb x$.

We shall denote by $\SFT$ the  solution set of the equation
\begin{equation}
\dot{z}\in\e F_{T}\left(t,z\right),\ \ z\left(0\right)=x_{0}\label{eq:T-inc}
\end{equation}
in $\ze$.

The following lemma can be easily verified.
\begin{lem}
If $\Fxt$ satisfies Assumption \ref{ass:Averaging condition} then
so does $F_{T}\left(t,x\right)$, for every $T>0$.
\end{lem}
The following result is our key lemma.
\begin{lem}
\label{lem:key_lemma}Suppose $\Fxt$ satisfies Assumption \ref{ass:Averaging condition}
and $T>0$ then
\[
d_{H}\left(\SF,\SFT\right)\le\e MT\left(1+\frac{3}{2}Ke^{K}\right).
\]
\end{lem}
\begin{proof}
Let $\x$ be an arbitrary solution of \eqref{eq:diffinc} in $\ze$
which we extend, in an arbitrary manner, to a solution of \eqref{eq:diffinc}
in $\left[0,\e^{-1}+T\right]$. We approximate $\x$ in $\ze$ by
$\tilde{x}\left(t\right)=\frac{1}{T}\int_{0}^{T}x^{*}\left(t+s\right)ds$.
This approximation satisfies 
\[
\left|\tilde{x}\left(t\right)-x^{*}\left(t\right)\right|\le\frac{1}{T}\int_{0}^{T}\left|x^{*}\left(t+s\right)-x^{*}\left(t\right)\right|ds\le\frac{\e}{T}\int_{0}^{T}Msds\le\frac{1}{2}\e MT,
\]
 for every $t\in\ze$. 

Since $\dot{\tilde{x}}\left(t\right)=\frac{1}{T}\int_{0}^{T}\dot{x}^{*}\left(t+s\right)ds$
the triangle inequality and the Lipschitz continuity of $\Fxt$ imply
that for every $t\in\ze$ 
\begin{eqnarray*}
d\left(\dot{\tilde{x}}\left(t\right),\e F_{T}\left(t,x^{*}\left(t\right)\right)\right) & \le & \e d_{H}\left(\frac{1}{T}\int_{0}^{T}F\left(t+s,x^{*}\left(t+s\right)\right)ds,\frac{1}{T}\int_{0}^{T}F\left(t+s,x^{*}\left(t\right)\right)ds\right)\\
 & \le & \frac{\e K}{T}\int_{0}^{T}\left|x^{*}\left(t+s\right)-x^{*}\left(t\right)\right|ds\le\frac{1}{2}\e^{2}KMT.
\end{eqnarray*}
Thus,
\begin{eqnarray*}
d\left(\dot{\tilde{x}}\left(t\right),\e F_{T}\left(t,\tilde{x}\left(t\right)\right)\right) & \le & d\left(\dot{\tilde{x}}\left(t\right),\e F_{T}\left(t,\xx t\right)\right)+d_{H}\left(\e F_{T}\left(t,\xx t\right),\e F_{T}\left(t,\tilde{x}\left(t\right)\right)\right)\\
 & \le & \frac{1}{2}\e^{2}KMT+\frac{1}{2}\e^{2}KMT=\e^{2}KMT,
\end{eqnarray*}
and
\[
\int_{0}^{\e^{-1}}d\left(\dot{\tilde{x}}\left(s\right),\e F_{T}\left(s,\tilde{x}\left(s\right)\right)\right)ds\le\e KMT.
\]
By the Filippov-Gronwall inequality (Theorem \ref{thm:Filippov lemma})
there exists $z^{**}\left(\cdot\right)$ a solution of \eqref{eq:T-inc}
which is $\e KMTe^{K}$ close to $\tilde{x}\pd$, hence, it is also
$\e MT\left(\frac{1}{2}+Ke^{K}\right)$ close to $\x$ in $\ze$.

On the other hand, let $z^{*}\pd$ be an arbitrary solution of \eqref{eq:T-inc}
defined on $\ze$. Now for every $t\in\ze$
\[
z^{*}\left(t\right)\in\int_{0}^{t}\e F_{T}\left(s,z^{*}\left(s\right)\right)ds=\frac{1}{T}\int_{0}^{t}\int_{0}^{T}\e F\left(s_{1}+s_{2},z^{*}\left(s_{1}\right)\right)ds_{2}ds_{1}
\]
Let $A=\left\{ \left(s_{1},s_{2}\right)\subset\R^{2}|s_{1}\in\left[0,\e^{-1}\right],s_{2}\in\left[s_{1},s_{1}+T\right]\right\} $,
and $u\left(s_{1},s_{2}\right)\in\e F\left(s_{2},z^{*}\left(s_{1}\right)\right)$
be a measurable selection defined almost everywhere in $A$, so that
for every $t\in\ze$ 
\[
z^{*}\left(t\right)=x_{0}+\frac{1}{T}\int_{0}^{t}\int_{s_{1}}^{s_{1}+T}u\left(s_{1},s_{2}\right)ds_{2}ds_{1}.
\]

To generate an approximation of $z^{*}\pd$, we extend it to $\left[-T,0\right]$
by setting $z^{*}\left(t\right)=x_{0}$, and extending $u\left(s_{1},s_{2}\right)$
to $\left[-T,0\right]\times\left[0,T\right]$ by choosing an arbitrary
measurably selection $u\left(s_{1},s_{2}\right)\in\e F\left(s_{2},x_{0}\right)$.
Then we approximate $z^{*}\left(t\right)$ by 
\[
\tilde{z}\left(t\right)=x_{0}+\frac{1}{T}\int_{0}^{t}\int_{s_{2}-T}^{s_{2}}u\left(s_{1},s_{2}\right)ds_{1}ds_{2}.
\]

Setting 
\[
A_{t}^{1}=\left\{ \left(s_{1},s_{2}\right)\subset\R^{2}|s_{1}\in\left[0,t\right],s_{2}\in\left[s_{1},s_{1}+T\right]\right\} 
\]
 
\[
A_{t}^{2}=\left\{ \left(s_{1},s_{2}\right)\subset\R^{2}|s_{2}\in\left[0,t\right],s_{1}\in\left[s_{2}-T,s_{2}\right]\right\} ,
\]
we express $z^{*}\pd$ and its approximation $\tilde{z}\pd$ by 
\[
z^{*}\left(t\right)=x_{0}+\frac{1}{T}\int_{A_{t}^{1}}u\left(s_{1},s_{2}\right)d\left(s_{1},s_{2}\right)
\]
\[
\tilde{z}\left(t\right)=x_{0}+\frac{1}{T}\int_{A_{t}^{2}}u\left(s_{1},s_{2}\right)d\left(s_{1},s_{2}\right).
\]
With this definition, we bound their difference for every $t\in\ze$
by 
\begin{eqnarray}
\left|\tilde{z}\left(t\right)-z^{*}\left(t\right)\right| & = & \frac{1}{T}\left|\int_{A_{t}^{1}}u\left(s_{1},s_{2}\right)d\left(s_{1},s_{2}\right)-\int_{A_{t}^{2}}u\left(s_{1},s_{2}\right)d\left(s_{1},s_{2}\right)\right|\label{eq:z_tilda_app}\\
 & = & \frac{1}{T}\left|\int_{A_{t}^{1}\backslash A_{t}^{2}}u\left(s_{1},s_{2}\right)d\left(s_{1},s_{2}\right)-\int_{A_{t}^{2}\backslash A_{t}^{1}}u\left(s_{1},s_{2}\right)d\left(s_{1},s_{2}\right)\right|\le\e MT,\nonumber 
\end{eqnarray}
 since $u\left(\cdot,\cdot\right)$ is bounded in norm by $\e M$
and the measure of the set $\left(A_{t}^{1}\backslash A_{t}^{2}\right)\cup\left(A_{t}^{2}\backslash A_{t}^{1}\right)$
is bounded by $T^{2}$.

To apply the Filippov-Gronwall inequality we bound 
\begin{equation}
d\left(\dot{\tilde{z}}\left(t\right),\e F\left(t,\tilde{z}\left(t\right)\right)\right)\le d\left(\dot{\tilde{z}}\left(t\right),\e F\left(t,z^{*}\left(t\right)\right)\right)+\e d_{H}\left(F\left(t,z^{*}\left(t\right)\right),F\left(t,\tilde{z}\left(t\right)\right)\right).\label{eq:integrand_fil_bnd}
\end{equation}

The second term above is bounded using \eqref{eq:z_tilda_app} by
$\e^{2}KMT$, and the first term above is bounded by
\begin{eqnarray*}
d\left(\dot{\tilde{z}}\left(t\right),\e F\left(t,z^{*}\left(t\right)\right)\right) & = & \e d_{H}\left(\frac{1}{T}\int_{0}^{T}F\left(t,z^{*}\left(t-s\right)\right)ds,F\left(t,z^{*}\left(t\right)\right)\right)\\
 & \le & \frac{\e K}{T}\int_{0}^{T}\left|z^{*}\left(t-s\right)-z^{*}\left(t\right)\right|ds\le\frac{1}{2}\e^{2}KMT,
\end{eqnarray*}
where we use the fact that
\[
\dot{\tilde{z}}\left(t\right)=\frac{1}{T}\int_{t-T}^{t}u\left(s,t\right)ds\in\frac{1}{T}\int_{t-T}^{t}\e F\left(t,z^{*}\left(s\right)\right)ds=\frac{1}{T}\int_{0}^{T}\e F\left(t,z^{*}\left(t-s\right)\right)ds.
\]

This bounds \eqref{eq:integrand_fil_bnd} by $\frac{3}{2}\e^{2}KMT$,
and establishes the existence of a solution a solution $x^{**}\pd$
of \eqref{eq:diffinc} which is $\frac{3}{2}\e KMTe^{K}$ far from
$\tilde{z}\pd$, hence, $\e MT\left(1+\frac{3}{2}Ke^{K}\right)$ far
from $z^{*}\pd$ in $\ze$. This completes the proof.
\end{proof}

\section{\label{sec:Averaging-differential-Inclusion}Averaging differential
Inclusions}

In this section we establish new estimates for the averaging of differential
inclusions in a general setting, which we apply to obtain sharp bounds
for multi-frequency differential inclusion. Specifically, We consider
two types of inclusions of the form \eqref{eq:diffinc}, where $\Fxt$
is either of the form $\Fxt=F_{1}\left(t,x\right)+F_{2}\left(t,x\right)+\dots+F_{m}\left(t,x\right)$
and each $\Fj$ is periodic in $t$ with period $T_{j}$, or when
each entry of $\Fxt$ is periodic, namely, $\Fxt=\left(F_{1}\left(t,x\right),F_{2}\left(t,x\right),\dots,F_{m}\left(t,x\right)\right)$,
and $\Fj$ is periodic in $t$ with period $T_{j}$.

Artstein \cite{artstein2007averaging} presented a new approach for
estimating the approximation error in the study of averaging ordinary
differential equations, which uses quantitative information on the
local fluctuations of the time-dependent vector field. He extended
this approach to control systems and differential inclusions in a
series of talks. 

We start by presenting Artstein's gauge; we then establish its additivity
and verify our main results on multi-frequency differential inclusions.
\begin{thm}
\label{thm:Art_bnd}Suppose $\Fxt$ satisfies Assumption \ref{ass:Averaging condition}
and there exists $\left(\Delta\left(\epsilon\right),\eta(\epsilon)\right)$
satisfying 
\begin{equation}
d_{H}\left(\frac{\epsilon}{\Delta(\epsilon)}\int_{s_{0}}^{s_{0}+\frac{\Delta\left(\epsilon\right)}{\epsilon}}\Fx sds,\bar{F}\left(x\right)\right)\leq\eta(\epsilon),\label{eq:Artstein fluctuation bounds-1}
\end{equation}
for all $s_{0}\geq0$ and $x\in\Omega$. Then the approximation error
of equation \eqref{eq:diffinc} is bounded by
\[
M\left(1+\frac{3}{2}Ke^{K}\right)\de+e^{K}\ee
\]
in the time interval $\left[0,\epsilon^{-1}\right]$. In particular
the approximation error is of order $O\left(\max\left(\Delta\left(\epsilon\right),\eta\left(\epsilon\right)\right)\right)$. \end{thm}
\begin{proof}
Set $T=\frac{\de}{\e}$. The triangle inequality bounds $d_{H}\left(\SF,\SFB\right)$
by 
\[
d_{H}\left(\SF,\SFT\right)+d_{H}\left(\SFT,\SFB\right).
\]
The first term above is bounded using Lemma \ref{lem:key_lemma} by
$M\left(1+\frac{3}{2}Ke^{K}\right)\de$ and the second term is bounded
using Corollary \ref{cor:filippov} by $\ee e^{K}$, since using our
definition of $T$ \eqref{eq:Artstein fluctuation bounds-1} is equivalent
to $d_{H}\left(F_{T}\left(t,x\right),\Fb x\right)\le\ee$.
\end{proof}
Applying this theorem to a periodic differential inclusion implies
the classical result of Plotnikov \cite{plotnikov1979averaging}.
\begin{cor}
Suppose $\Fxt$ satisfies Assumption \ref{ass:Averaging condition}
and that it is periodic in $t$ with period $T$. Then the approximation
error of equation \eqref{eq:diffinc} is bounded by $M\left(1+\frac{3}{2}Ke^{K}\right)T\e$.
In particular it is of order $O\left(\e\right)$.\end{cor}
\begin{proof}
Set $\de=\epsilon T$ and $\ee=0$ and apply Theorem \ref{thm:Art_bnd}.

To prove the next corollary we need the following lemma.\end{proof}
\begin{lem}
\label{lem:cnv_lem-1}Suppose $D_{1},D_{2},E_{1},E_{2}\subset\Rd$
are non-empty convex sets. Then 
\[
d_{H}\left(D_{2},E_{2}\right)\le d_{H}\left(D_{1}+D_{2},E_{1}+E_{2}\right)+d_{H}\left(D_{1},E_{1}\right).
\]
\end{lem}
\begin{proof}
By the properties of the support function (see, \cite[Theorem 1.7.5]{schneider1993convex})
we express the Hausdorff distance $d_{H}\left(D_{1}+D_{2},E_{1}+E_{2}\right)$
by
\[
\left\Vert h_{D_{1}+D_{2}}\left(x\right)-h_{E_{1}+E_{2}}\left(x\right)\right\Vert =\left\Vert h_{D_{1}}\left(x\right)+h_{D_{2}}\left(x\right)-h_{E_{1}}\left(x\right)-h_{E_{2}}\left(x\right)\right\Vert .
\]
The reverse triangle inequality bounds the latter from below by
\[
\left|\left\Vert h_{D_{1}}\left(x\right)-h_{E_{1}}\left(x\right)\right\Vert -\left\Vert h_{D_{2}}\left(x\right)-h_{E_{2}}\left(x\right)\right\Vert \right|\ge\left\Vert h_{D_{1}}\left(x\right)-h_{E_{1}}\left(x\right)\right\Vert -\left\Vert h_{D_{2}}\left(x\right)-h_{E_{2}}\left(x\right)\right\Vert ,
\]
which completes the proof.
\end{proof}
The following corollary extends the classical estimate in \cite[Theorem 4.3.6]{sanders2007averaging}
from differential equations to differential inclusions.
\begin{cor}
Suppose $\Fxt$ satisfies Assumption \ref{ass:Averaging condition}
and 
\begin{equation}
\sup_{x\in\D,T\in\ze}\e d_{H}\left(\int_{0}^{T}F\left(s,x\right)ds,T\bar{F}\left(x\right)\right)\le\delta\left(\e\right).\label{eq:delta_bnd}
\end{equation}
Then the approximation error of equation \eqref{eq:diffinc} is of
order $O\left(\sqrt{\delta\left(\e\right)}\right)$.\end{cor}
\begin{proof}
Dividing both sides of Inequality \eqref{eq:delta_bnd} by $\e$,
one obtains
\[
\sup_{x\in\D,T\in\ze}d_{H}\left(\int_{0}^{T}F\left(s,x\right)ds,T\bar{F}\left(x\right)\right)\le\frac{\delta\left(\e\right)}{\e}.
\]

Let us fix $x\in\D$, $T=\frac{\sqrt{\delta\left(e\right)}}{\e}$
and $s_{0}\in\left[0,\e^{-1}-T\right]$. Applying Lemma \ref{lem:cnv_lem-1}
with the convex sets $D_{1}=\int_{0}^{s_{0}}F\left(s,x\right)ds,\ D_{2}=\int_{s_{0}}^{s_{0}+T},F\left(s,x\right)ds,E_{1}=s_{0}\overline{F}\left(x\right)$
and $E_{2}=T\overline{F}\left(x\right)$ we bound
\begin{equation}
d_{H}\left(\int_{s_{0}}^{s_{0}+T}F\left(s,x\right)ds,T\bar{F}\left(x\right)\right)\label{eq:int_s0_T}
\end{equation}
using \eqref{eq:delta_bnd} by 
\[
d_{H}\left(\int_{0}^{s_{0}}F\left(s,x\right)ds,s_{0}\bar{F}\left(x\right)\right)+d_{H}\left(\int_{0}^{s_{0}+T}F\left(s,x\right)ds,\left(s_{0}+T\right)\bar{F}\left(x\right)\right)\le2\frac{\delta\left(\e\right)}{\e}.
\]
Dividing Equation \eqref{eq:int_s0_T} by $T$ yields
\[
d_{H}\left(\frac{1}{T}\int_{s_{0}}^{s_{0}+T}F\left(s,x\right)ds,\bar{F}\left(x\right)\right)\le2\frac{\delta\left(\e\right)}{\e T}=2\sqrt{\delta\left(\e\right)}.
\]
Since $x$ and $s_{0}$ are arbitrary, the conditions of Theorem \ref{thm:Art_bnd}
are satisfied in $\left[0,\left(1-\sqrt{\delta\left(\e\right)}\right)\e^{-1}\right]$
with $\de=\sqrt{\delta\left(\e\right)}$ and $\ee=2\sqrt{\delta\left(\e\right)}$,
and it is easy to see that a bound of order $O\left(\sqrt{\delta\left(\e\right)}\right)$
follows.
\end{proof}
Following \cite{bright2009tight} we provide an additivity property
for this bound namely, we show that when we can express $\Fxt=F_{1}\left(t,x\right)+F_{2}\left(t,x\right)+\dots+F_{m}\left(t,x\right)$,
then the estimate of the sum is bounded by the sum of the estimates.
This result is then applied to multi-frequency systems.
\begin{thm}
\label{thm:(Additivity-Lemma)-Given-1}Suppose $\Fxt=F_{1}\left(t,x\right)+F_{2}\left(t,x\right)+\dots+F_{m}\left(t,x\right)$
satisfies Assumption \ref{ass:Averaging condition}, and that for
every $j=1,\dots m$ the set-valued mapping $\Fj$ has a well defined
average $\bar{F}_{j}\left(x\right)$. If for every $\jm$ there exist
$\left(\Delta_{j}(\epsilon),\eta_{j}(\epsilon)\right)$ satisfying
\begin{equation}
\left|d_{H}\left(\frac{\epsilon}{\Delta_{j}(\epsilon)}\int_{s_{0}}^{s_{0}+\frac{\Delta_{j}\left(\epsilon\right)}{\epsilon}}F_{j}\left(s,x\right)ds,\bar{F}_{j}\left(x\right)\right)\right|\leq\eej,\label{eq:art_bnd_i}
\end{equation}
for all $s_{0}\geq0$ and $x\in\D$. Then the approximation error
of equation \eqref{eq:diffinc} is bounded by 
\begin{equation}
M\left(1+\frac{3}{2}Ke^{K}\right)\smj\dej+e^{K}\smj\eej.\label{eq:bnd_sum}
\end{equation}
\end{thm}
\begin{proof}
To verify this theorem we shall use the following two observations.
Suppose $G\left(t,x\right)$ satisfies Assumption \ref{ass:Averaging condition}
and $\left|d_{H}\left(G\left(s,x\right),\bar{G}\left(x\right)\right)\right|\leq\alpha$,
then $\left|d_{H}\left(G_{T}\left(s,x\right),\bar{G}\left(x\right)\right)\right|\leq\alpha$
for every $T>0$. Also, Fubini's theorem implies that 
\[
\frac{1}{T_{2}}\int_{t}^{t+T_{2}}G_{T_{1}}\left(s,x\right)ds=\frac{1}{T_{1}}\int_{t}^{t+T_{1}}G_{T_{2}}\left(s,x\right)ds,
\]
for every $T_{1},T_{2}>0$, $x\in\D$ and $t\in\R$.

For every $\jm$ we set t $T_{j}=\frac{\dej}{\e}$ and define a sequence
of set-valued mapping, by setting $F_{j}^{0}\left(t,x\right)=\Fj$
and 
\[
F_{j}^{i}\left(t,x\right)=\frac{1}{T_{i}}\int_{t}^{t+T_{i}}F_{j}^{i-1}\left(s,x\right)ds,
\]
for $i=1,\dots,m$. We also set $F^{0}\left(t,x\right)=\Fxt$ and
$F^{i}\left(t,x\right)=\smj F_{j}^{i}\left(t,x\right)$. These set-valued
mappings satisfy $F^{i}\left(t,x\right)=F_{T_{i}}^{i-1}\left(t,x\right)$.
Now by the triangle inequality we bound the approximation error, given
by $d_{H}\left(\SF,\SFB\right)$, by
\begin{equation}
d_{H}\left(\S{F^{m}},\SFB\right)+\smj d_{H}\left(\S{F^{i-1}},\S{F^{i}}\right)\label{eq:gen_add_bnd}
\end{equation}
From the aforementioned observations we conclude that for every $\jm$
and $t\in\ze$ we have that 
\[
d_{H}\left(F_{j}^{m}\left(t,x\right),\bar{F}_{j}\left(x\right)\right)\le d_{H}\left(\frac{1}{T_{j}}\int_{s_{0}}^{s_{0}+T_{j}}F_{j}\left(s,x\right)ds,\bar{F}_{j}\left(x\right)\right)\le\eej.
\]
Thus, we bound $d_{H}\left(F^{m}\left(t,x\right),\bar{F}\left(x\right)\right)\le\smj\eej$
and Corollary \ref{cor:filippov} yields 
\[
d_{H}\left(\S{F^{m}},\S{\bar{F}}\right)\le\smj\eej.
\]
Applying Lemma \ref{lem:key_lemma} to each term in the sum, we conclude
that \eqref{eq:gen_add_bnd} is bounded by $ $
\[
\smj\left(M\left(1+\frac{3}{2}Ke^{K}\right)\dej+e^{K}\eej\right).
\]

\end{proof}
This latter theorem implies one of our main results.
\begin{cor}
\label{cor:multi_freq_sum}Suppose $\Fxt=F_{1}\left(t,x\right)+F_{2}\left(t,x\right)+\dots+F_{m}\left(t,x\right)$
satisfies Assumption \ref{ass:Averaging condition} and for every
$j=1,\dots,m$ the set-valued mapping $\Fj$ is periodic in $t$ with
period $T_{j}$. Then the approximation error of equation \eqref{eq:diffinc}
is bounded by 
\begin{equation}
\e M\left(1+\frac{3}{2}Ke^{K}\right)\smj T_{j}.\label{eq:multi_periodic_bnd}
\end{equation}
In particular the estimation is of order $O\left(\e\right)$.\end{cor}
\begin{proof}
For every $\jm$ set $\dej=\e T_{j}$ and $\eej=0$, and apply Theorem
\ref{thm:(Additivity-Lemma)-Given-1}.
\end{proof}
We now extend our result to multi-frequency differential inclusions
where $\Fxt$ is of the form $\Fxt=\left(F_{1}\left(t,x\right),F_{2}\left(t,x\right),\dots,F_{d}\left(t,x\right)\right)$,
where each of its components ($\Fj$) satisfies a bound of the form
\eqref{eq:art_bnd_i}. This extension is crucial in the following
section where our results are applied to control systems, since Theorem
\ref{thm:(Additivity-Lemma)-Given-1} cannot be applied to such set-valued
mappings, as they might not be representable by a sum of periodic
set-valued mappings. This can be seen in the following example.
\begin{example}
\label{ex:non_dec}Let $U=\left[0,1\right]\times\left[0,2\pi\right]$.
The set-valued mapping
\[
F\left(t,x\right)=F\left(t\right)=\left\{ \left(7\cos t+u_{1}\cos u_{2},7\sin\pi t+u_{1}\sin u_{2}\right)\in\R^{2}|\left(u_{1},u_{2}\right)\in U\right\} ,
\]
is not the Minkovski sum of its components, namely, 
\[
F\left(t,x\right)\neq\left\{ \left(7\cos t+u_{1}\cos u_{2},0\right)|\left(u_{1},u_{2}\right)\in U\right\} +\left\{ \left(0,7\sin\pi t+u_{1}\sin u_{2}\right)|\left(u_{1},u_{2}\right)\in U\right\} ,
\]
and it is easy to see that there is no such representation.

Applying the same line of proof as in Theorem \ref{thm:(Additivity-Lemma)-Given-1}
we conclude the following result.\end{example}
\begin{thm}
\label{thm:nD_bnd}Suppose $\Fxt=\left(F_{1}\left(t,x\right),F_{2}\left(t,x\right),\dots,F_{d}\left(t,x\right)\right)$
satisfies Assumption \ref{ass:Averaging condition}, and for every
$j=1,\dots,m$ the set-valued mapping $\Fj$ has a well defined average
$\bar{F}_{j}\left(x\right)$. If for every $\jm$ there exists $\left(\dej,\eej\right)$
satisfying \ref{eq:art_bnd_i}, then the approximation error of equation
\eqref{eq:diffinc} is bounded by 
\[
M\left(1+\frac{3}{2}Ke^{K}\right)\sum_{j=1}^{d}\dej+e^{K}\sqrt{\sum_{j=1}^{d}\left(\eej\right)^{2}}.
\]
In particular the estimation is of order $O\left(\e\right)$.\end{thm}
\begin{proof}
The proof follows from the proof of Theorem \ref{thm:(Additivity-Lemma)-Given-1},
with the exception here we use the properties of the Euclidian norm
to bound 
\[
d_{H}\left(F^{d}\left(t,x\right),\bar{F}\left(x\right)\right)\le\sqrt{\sum_{j=1}^{d}\left[d_{H}\left(F_{j}^{d}\left(t,x\right),\bar{F}_{j}\left(x\right)\right)\right]^{2}}\le\sqrt{\sum_{j=1}^{d}\left(\eej\right)^{2}}.
\]

\end{proof}
This result immediately implies the following corollary on averaging
inclusions each of entry of $\Fxt$ has a different period.
\begin{cor}
\label{cor:multi_frec_entry_diffinc}Suppose $\Fxt=\left(F_{1}\left(t,x\right),F_{2}\left(t,x\right),\dots,F_{d}\left(t,x\right)\right)$
satisfies Assumption \ref{ass:Averaging condition}, and that for
every $j=1,\dots,m$ its $j$'th entry $\Fj$ is periodic in $t$
with period $T_{j}$. Then the approximation error of equation \eqref{eq:diffinc}
is bounded by \eqref{eq:multi_periodic_bnd} with $m=d$. In particular
the estimation is of order $O\left(\e\right)$.
\end{cor}
One can also conclude from the proof of Theorem \ref{thm:(Additivity-Lemma)-Given-1}
the following estimates which, in some cases, may improve the estimates
from Corollaries \ref{cor:multi_freq_sum} and \ref{cor:multi_frec_entry_diffinc}.
\begin{cor}
\label{cor:multi_freq_dilute}Suppose that either the conditions of
Corollary \ref{cor:multi_freq_sum} hold and
\[
\left\{ T_{1},T_{2},\dots,T_{m}\right\} \subset\left\{ T_{1},T_{2},\dots,T_{N}\right\} ,
\]
or the conditions of Corollary \ref{cor:multi_frec_entry_diffinc}
hold and
\[
\left\{ T_{1},T_{2},\dots,T_{d}\right\} \subset\left\{ T_{1},T_{2},\dots,T_{N}\right\} .
\]
Then the approximation error of equation \eqref{eq:diffinc} is bounded
by \eqref{eq:multi_periodic_bnd} with $m=N$.
\end{cor}

\begin{cor}
\label{cor:per_full_gen}Suppose $\Fxt=F_{1}\left(t,x\right)+F_{2}\left(t,x\right)+\dots+F_{m}\left(t,x\right)$
satisfies Assumption \ref{ass:Averaging condition} and for every
$j=1,\dots,m$ the $i$'th component of $\Fj$ is periodic in $t$
with period $T_{j,i}$. If 
\[
\left\{ T_{i,j}|\jm,i=1,\dots,d\right\} \subset\left\{ T_{1},\dots,T_{N}\right\} 
\]
then the approximation error of equation \eqref{eq:diffinc} is bounded
by \eqref{eq:multi_periodic_bnd} with $m=N$. In particular, it is
of order $O\left(\e\right)$.
\end{cor}

\section{\label{sec:Application-in-multi-frequency}Application in multi-frequency
Control Systems}

We now present an application of the bound obtained in the previous
section to control systems and establish an $O\left(\e\right)$ bound
for multi-frequency control systems. Following which, we provide an
example of our main result.

We shall consider the approximation of the solution-set of the control
system of the form
\begin{equation}
\dot{x}=\e g\left(t,x,u\right),\; x\left(0\right)=x_{0},\label{eq:multi-freq_cont-1}
\end{equation}
where 
\[
g\left(t,x,u\right)=g_{1}\left(t,x,u\right)+g_{2}\left(t,x,u\right)+\cdots+g_{m}\left(t,x,u\right),
\]
and each $g_{j}\left(t,x,u\right)$ is periodic in $t$ with period
$T_{j}$. We denote its solution set by $\S G$, which we approximate
by $\S{\bar{G}}$, the solution set of the corresponding averaged
system
\begin{equation}
\dot{y}\in\e\bar{G}\left(y\right)=\lim_{T\rightarrow\infty}\frac{\e}{T}\int_{0}^{T}G\left(y,s\right)ds=\lim_{T\rightarrow\infty}\frac{\e}{T}\int_{0}^{T}\left\{ g\left(s,y,u\right)|u\in U\right\} ds,\label{eq:cnt_avg}
\end{equation}
where we define
\[
G\left(t,x\right)=\left\{ g\left(t,x,u\right)|u\in U\right\} .
\]

We assume our system satisfies the following conditions.
\begin{assumption}
\label{ass:control}We assume that $U\subset\mathbb{R}^{k}$ is compact
and that for every $\jm$ the following conditions holds:
\begin{enumerate}
\item $g_{j}\left(t,x,u\right):\mathbb{R}\times\D\times U\rightarrow\Rd$
is bounded in norm by $M_{j}$.
\item $g_{j}\left(t,x,u\right)$ is measurable in $t$ and $u$.
\item $g_{j}\left(t,x,u\right)$ satisfies Lipschitz conditions in $x$
uniformly in $t$ and $u$, with a Lipschitz constant $K_{j}$.
\end{enumerate}
\end{assumption}
Our assumptions imply that $G\left(t,x\right)$ satisfies Assumption
\eqref{ass:Averaging condition} (expect for being convex valued,
which by Remark \ref{rem: non-cnv} is not essential in our proofs),
with $M=\smj M_{j}$ and $K=\smj K_{j}$, where the periodicity of
the functions $g_{j}\left(t,x,u\right)$ implies that the average
of $G\left(t,x\right)$ exists.

Our main result for this section is as follows.
\begin{thm}
\label{thm:multi_freq_cnt}Suppose $g(t,x,u)=g_{1}\left(t,x,u\right)+g_{2}\left(t,x,u\right)+\cdots+g_{m}\left(t,x,u\right)$
satisfies the condition of Assumption \ref{ass:control}, and for
every $\jm$ the function $g_{j}\left(t,x,u\right)$ is periodic in
$t$ with period $T_{j}$. Then the approximation error of \eqref{eq:multi-freq_cont-1}
is bounded by
\[
\e\sqrt{m}M_{H}\left(1+\frac{3}{2}K_{H}e^{K_{H}}\right)\sum_{j=1}^{N}T_{j},
\]
where $M_{H}=\sqrt{\smj M_{j}^{2}}$ and $K_{H}=\sqrt{m\smj K_{j}^{2}}$.
In particular, it is of order $O\left(\e\right)$.
\end{thm}
The difficulty in applying our results to \eqref{eq:multi-freq_cont-1}
lies in the fact that all the $g_{j}\left(t,x,u\right)$'s employ
the same control $u$, thus one cannot necessarily write $G\left(t,x\right)$
as a sum of periodic set-valued mapping (see Example \ref{ex:non_dec}),
and our results cannot trivially be applied. Instead, we shall ``decouple''
the periods in the system, by splitting the multi-frequency system
to a system of $m$ coupled periodic equations, each having a different
period, to which we apply our bounds. Our auxiliary system is of dimension
$md$, and it contains $m$ subsystems, each of dimension $d$ having
a periodic vector field. In order that this system represents the
solution of the original equation, we must couple all the the new
variables. 

We represent a vector in $\R^{md}$ by $z=\left(z_{1},z_{2},\dots,z_{m}\right)\in\R^{md}$
where $z_{j}\in\Rd$, and we define the linear map 
\[
\Phi\left(z\right)=\smj z_{j},
\]
which is Lipschitz continuous with a Lipschitz constant $\sqrt{m}$.
With this notation we defined the auxiliary system $\dot{z}=\e h\left(t,z,u\right),\ z\left(0\right)=\left(x_{0},\mathbf{0},\dots,\mathbf{0}\right)$
by 
\begin{equation}
\dot{z}=\e\left[\begin{array}{c}
\dot{z}_{1}\\
\dot{z}_{2}\\
\vdots\\
\dot{z}_{m}
\end{array}\right]=\e\left[\begin{array}{c}
g_{1}\left(t,z_{1}+\cdots+z_{m},u\right)\\
g_{2}\left(t,z_{1}+\cdots+z_{m},u\right)\\
\vdots\\
g_{m}\left(t,z_{1}+\cdots+z_{m},u\right)
\end{array}\right]=\e\left[\begin{array}{c}
g_{1}\left(t,\Phi\left(z\right),u\right)\\
g_{2}\left(t,\Phi\left(z\right),u\right)\\
\vdots\\
g_{m}\left(t,\Phi\left(z\right),u\right)
\end{array}\right]=\e h\left(t,z,u\right).\label{eq:decop_eq}
\end{equation}

This system is constructed so that $g\left(t,\Phi\left(z\right),u\right)=\Phi\left(h\left(t,z,u\right)\right)$,
which also holds for the corresponding averaged systems, namely, $\bar{G}\left(\Phi\left(z\right)\right)=\Phi\left(\bar{H}\left(z\right)\right)$,
where 
\begin{equation}
\bar{H}(x)=\lim_{T\rightarrow\infty}\frac{1}{T}\int_{0}^{T}\left\{ h(s,x,u)|u\in U\right\} ds.\label{eq:decop_cont_averaged_eq}
\end{equation}

These equalities imply that if $x^{*}\left(\cdot\right)$ is a solution
of \eqref{eq:multi-freq_cont-1} with control $u^{*}\left(\cdot\right)$,
then applying the same control to the auxiliary system \eqref{eq:decop_eq},
we obtain a solution $z^{*}\left(\cdot\right)$ satisfying 
\[
x^{*}\left(t\right)=\Phi\left(z^{*}\left(t\right)\right)={z^{*}}_{1}\left(t\right)+\cdots+{z^{*}}_{m}\left(t\right),
\]
for every $t\in\ze$; and vice-versa. Thus we conclude that $d_{H}\left(\S G,\Phi\left(\S H\right)\right)=0$,
and similarly that $d_{H}\left(\S{\bar{G}},\Phi\left(\S{\bar{H}}\right)\right)=0$. 

The approximation error of averaging the auxiliary system \eqref{eq:decop_eq}
is bounded as follows.
\begin{lem}
\label{lem:decop_sys}The approximation error of \eqref{eq:decop_eq}
is bounded by 
\[
\e M_{H}\left(1+\frac{3}{2}K_{H}e^{K_{H}}\right)\sum_{j=1}^{m}T_{j},
\]
where $M_{H}=\sqrt{\smj M_{j}^{2}}$ and $K_{H}=\sqrt{m\smj K_{j}^{2}}$.\end{lem}
\begin{proof}
It is clear that the function $h\left(t,z,u\right)$ is bounded in
norm by $M_{H}$. To estimate its Lipschitz condition we observe that
for arbitrary $z^{1},z^{2}\in\D$, 
\[
\left|g_{j}\left(t,\Phi\left(z^{1}\right),u\right)-g_{j}\left(t,\Phi\left(z^{2}\right),u\right)\right|\le K_{j}\left|\Phi\left(z^{1}\right)-\Phi\left(z^{2}\right)\right|\le\sqrt{m}K_{j}\left|z^{1}-z^{2}\right|.
\]
 Thus $K_{H}$ is a Lipschitz constant of $h\left(t,z,u\right)$,
and the lemma follows from Corollary \ref{cor:multi_freq_dilute}.
\end{proof}
We are now ready to prove the main result of this section.
\begin{proof}
[Proof of Theorem \ref{thm:multi_freq_cnt}] The triangle inequality
bounds $d_{H}\left(\S G,\S{\bar{G}}\right)$ by 
\[
d_{H}\left(\S G,\Phi\left(\S H\right)\right)+d_{H}\left(\Phi\left(\S H\right),\Phi\left(\S{\bar{H}}\right)\right)
\]
\[
+d_{H}\left(\Phi\left(\S{\bar{H}}\right),\S{\bar{G}}\right).\hspace{1em}\hspace{1em}\hspace{1em}\hspace{1em}\hspace{1em}\hspace{1em}
\]
While the first and third terms above equal zero, the second term
is bounded using the Lipschitz constant of $\Phi\left(\cdot\right)$
and Lemma \ref{lem:decop_sys} by
\begin{eqnarray*}
\\
d_{H}\left(\Phi\left(\S H\right),\Phi\left(\S{\bar{H}}\right)\right) & \le & \sqrt{m}d_{H}\left(\S H,\S{\bar{H}}\right)\\
 & \le & \e\sqrt{m}M_{H}\left(1+\frac{3}{2}K_{H}e^{K_{H}}\right)\sum_{j=1}^{m}T_{j}.
\end{eqnarray*}

\end{proof}
The latter theorem can be extended in a similar manner to Corollary
\ref{cor:per_full_gen} when each entry of $g_{j}\left(t,x,u\right)$
has a different period.
\begin{thm}
\label{thm:cnt_elem_per}Suppose $g(t,x,u)=g_{1}\left(t,x,u\right)+g_{2}\left(t,x,u\right)+\cdots+g_{m}\left(t,x,u\right)$
satisfies conditions 1-3 of Assumption \ref{ass:control}, and for
every $\jm$ the $i$'th entry of $g_{j}\left(t,x,u\right)$ is periodic
in $t$ with period $T_{j,i}$. If 
\[
\left\{ T_{i,j}|\jm,i=1,\dots,d\right\} \subset\left\{ T_{1},\dots,T_{N}\right\} 
\]
then the approximation error of \eqref{eq:multi-freq_cont-1} is $\e\sqrt{m}M_{H}\left(1+\frac{3}{2}K_{H}e^{K_{H}}\right)\sum_{j=1}^{N}T_{j}$,
where $M_{H}=\sqrt{\smj M_{j}^{2}}$ and $K_{H}=\sqrt{m\smj K_{j}^{2}}$.
In particular, it is of order $O\left(\e\right)$.
\end{thm}
The following is an application of our results.
\begin{example}
Consider the control system given by
\[
\dot{x}=\e g\left(t,x,u\right)=\e x+\e u\left(\cos\left(2\pi t\right)+\cos\left(2t\right)\right),\ x\left(0\right)=0.
\]
where $U=\left[-1,1\right]$. The averaged equation in this case can
be expressed, replacing the time average by a space average, by 
\begin{eqnarray*}
\dot{y}\in\e\bar{G}\left(y\right) & = & \left\{ \e y+\frac{1}{\left(2\pi\right)^{2}}\int_{\left[0,2\pi\right]^{2}}\e u\left(\phi_{1},\phi_{2}\right)\left(\cos\left(\phi_{1}\right)+\cos\left(\phi_{2}\right)\right)d\left(\phi_{1},\phi_{2}\right)|u:\left[0,2\pi\right]^{2}\rightarrow\left[-1,1\right]\right\} .
\end{eqnarray*}
The set $\overline{G}\left(y\right)$ is convex and by symmetry we
conclude that $\overline{G}\left(y\right)=\left[y-\alpha,y+\alpha\right]$,
where
\[
\alpha=\frac{1}{\left(2\pi\right)^{2}}\int_{\left[0,2\pi\right]^{2}}\left|\cos\left(\phi_{1}\right)+\cos\left(\phi_{2}\right)\right|d\phi\approx0.815
\]
was computed analytically. Notice that this does not agree with the
naive time averaging of the vector field which yields the function
$\bar{g}\left(x,u\right)=x$.

So in the the domain $\D=\left[-2,2\right]$ we have that $M_{H}=\sqrt{10}$,
$K_{H}=\sqrt{2}$ (by setting $g_{1}\left(t,x,u\right)=x+u\cos\left(2\pi t\right)$
and $g_{2}\left(t,x,u\right)=u\cos\left(2t\right)$) and our theorem
implies that the estimation error is bounded by 
\[
\e\sqrt{20}\left(1+\frac{3}{2}\sqrt{2}e^{\sqrt{2}}\right)\left(1+\pi^{-1}\right).
\]
\end{example}
\begin{acknowledgement*}
The author wishes to thank Zvi Artstein for his insightful remarks.
\end{acknowledgement*}
\bibliographystyle{plain}
\bibliography{bib}

\end{document}